\documentclass{amsart}

\usepackage{amssymb}

\theoremstyle{plain}

\numberwithin{equation}{section}

\newtheorem*{teorema}{Theorem A}
\newtheorem*{teoremab}{Theorem B}

\newtheorem{hyp}[equation]{Hypotheses}

\newtheorem{lem}[equation]{Lemma}

\newcommand{\trace}{\operatorname{Trace}}

\newcommand{\SL}{\operatorname{SL}}

\newcommand{\GL}{\operatorname{GL}}

\theoremstyle{definition}

\newtheorem{rem}[equation]{Remark}
\begin{document}	
\title{On conjugacy classes of $\GL(n,q)$ and $\SL(n,q)$}
\author{Edith Adan-Bante}

\address{Department of Mathematical Science, Northern Illinois University, 
 Watson Hall 320
DeKalb, IL 60115-2888, USA} 

\email{adanbant@math.niu.edu}

\author{John M. Harris}

\address{Department of Mathematics, University of Southern Mississippi, 
730 East Beach Boulevard, 
Long Beach, MS 39560, USA}

\email{john.m.harris@usm.edu}

\keywords{conjugacy classes, products of conjugacy classes, general linear group, matrices, special linear group}

\subjclass{20G40, 20E45}

\date{2009}
\begin{abstract} 
Let $\GL(n,q)$ be the group of $n\times n$ invertible matrices over a field with $q$ elements, and 
$\SL(n,q)$ be the group of $n\times n$ matrices with determinant 1 over a field with $q$ elements. 
We prove that the product of any two non-central conjugacy classes in $\GL(n,q)$ is the union of 
at least $q-1$ distinct conjugacy classes, and that 
 the product of any two non-central conjugacy classes in $\SL(n,q)$ is the union of 
at least $\left\lceil\frac{q}{2}\right\rceil$ distinct conjugacy classes.

\end{abstract}
\maketitle

\begin{section}{Introduction}
Let $\mathcal{G}$ be a finite group and  $A\in \mathcal{G}$. Denote by $A^{\mathcal{G}}=\{A^B\mid B\in \mathcal{G}\}$ 
 the conjugacy class of $A$ in $\mathcal{G}$.
Let $\mathcal{X}$ be a $\mathcal{G}$-invariant subset of $\mathcal{G}$, i.e. 
$\mathcal{X}^A=\{B^A\mid B\in \mathcal{X}\}=\mathcal{X}$ for all $A\in \mathcal{G}$.
  Then $\mathcal{X}$ can be expressed as a union of 
  $n$ distinct conjugacy classes of $\mathcal{G}$, for some integer $n>0$. Set
 $\eta(\mathcal{X})=n$.
 
 Given any conjugacy classes $A^{\mathcal{G}}$, $B^{\mathcal{G}}$ in $\mathcal{G}$, we can check that the 
 product $A^{\mathcal{G}} B^{\mathcal{G}}=\{XY\mid X\in A^{\mathcal{G}}, Y\in B^{\mathcal{G}}\}$ is a $\mathcal{G}$-invariant
 subset of $\mathcal{G}$ and 
 thus $A^{\mathcal{G}}B^{\mathcal{G}}$ is the union of $\eta(A^{\mathcal{G}}B^{\mathcal{G}})$ distinct conjugacy classes of 
 $\mathcal{G}$. For instance, if $A$ or $B$ is in the center ${\bf Z}(\mathcal{G})$ of $\mathcal{G}$, 
 then  $A^{\mathcal{G}}B^{\mathcal{G}}=(AB)^{\mathcal{G}}$ and thus $\eta( A^{\mathcal{G}}B^{\mathcal{G}})=1$. 
 
 Fix a prime $p$ and integers $m>0$ and $n\geq 2$. 
 Let $\mathcal{F} = \mathcal{F}(q)$ be a field with $q=p^m$ elements and $\mathcal{G}=\GL(n,q)$ be the general linear group
 of $n\times n$ invertible matrices over $\mathcal{F}$. 
 In this note, given any two conjugacy classes $A^{\mathcal{G}},B^{\mathcal{G}}$ of $\mathcal{G}$,
  we explore the relationship between $\eta(A^{\mathcal{G}}B^{\mathcal{G}})$ and $q$. 

 \begin{teorema} 
 Let   $A$ and $B$ be
 matrices in $\mathcal{G}=\GL(n, q)$. Then one of the following holds:
 
 (i)  $A^{\mathcal{G}}B^{\mathcal{G}}= (AB)^{\mathcal{G}}$ and at least one of $A$, $B$ is a scalar matrix.
 
   (ii) $A^{\mathcal{G}}B^{\mathcal{G}}$ is the union of at least
  $q-1$ distinct conjugacy classes, i.e. $\eta(A^{\mathcal{G}}B^{\mathcal{G}})\geq q-1$. 
 \end{teorema}

Given any group $G$, 
denote by $\min(G)$ the smallest integer in 
the set $\{\eta(a^G b^G)\mid a,b\in  G\setminus {\bf Z}(G)  \}$.
 We want to emphasize that the previous result is not an optimal result, that is, 
 $\min(\GL(n,q))>q-1$ for certain values of $n$ and $q$.
By Remark \ref{optimal}, we have that $\min(\GL(2,2^m))=2^m-1$ for any integer $m>1$. Also, using 
 GAP \cite{GAP4}, we can check that $\min(\GL(2,q))=q-1$  for $q\in \{ 3,5,7,9,11,13\}$, but $\GL(3,3)=4 >2$. 
  Hence, we suspect that $\min(\GL(n,q))$ should be a function of $n$ as well as $q$. 
  
  We now  turn our attention to the special linear group $\SL(n,q)$, the group of $n\times n$ matrices with determinant $1$ over a finite field with $q$ elements.
 \begin{teoremab} 
 Let   $A$ and $B$ be
 matrices in $\mathcal{S}=\SL(n, q)$. Then one of the following holds:
 
 (i)  $A^{\mathcal{S}}B^{\mathcal{S}}= (AB)^{\mathcal{S}}$ and at least one of $A$, $B$ is a scalar matrix.
 
   (ii) $A^{\mathcal{S}}B^{\mathcal{S}}$ is the union of at least
  $\left\lceil\frac{q}{2}\right\rceil$ distinct conjugacy classes, i.e. $\eta(A^{\mathcal{S}}B^{\mathcal{S}})\geq \left\lceil\frac{q}{2}\right\rceil$. 
 \end{teoremab}

For a detailed study of $\SL(2,q)$, see \cite{sl2q}.
 
{\bf Acknowledgment.}
We would like to thank Professor Everett C. Dade for his suggestion to study products of conjugacy
classes in the linear groups.
\end{section}

\begin{section}{Proofs}

\begin{lem}\label{interchange}
Let $G$ be a finite group, $a^G$ and $b^G$ be conjugacy classes of $G$. 
Then $a^G b^G= b^G a^G$.
\end{lem}
\begin{proof}
See Lemma 3 of \cite{symmetric}.
\end{proof}

\begin{lem}\label{fieldsize} 
Let $\mathcal{F}$ be a finite field with $q$ elements, 
and $a, b,c\in \mathcal{F}$ with $a\neq 0$. 

(i) If $q$ is even, then the set
$\{ai^2+ c\mid i\in \mathcal{F}\}$ has $q$ elements.

(ii)  The set
$\{ai^2+bi+c\mid i\in \mathcal{F}\}$ has at least $\left\lceil\frac{q}{2}\right\rceil$ elements.
 
\end{lem}
\begin{proof}
(i) If the field $\mathcal{F}$ has even characteristic, the map $x\mapsto x^2$ is an automorphism of the field.
Observe then
that  $$|\{ai^2+ c\mid i\in \mathcal{F}\}|= |\{ai^2\mid i\in \mathcal{F}\}|=|\{i^2\mid i\in \mathcal{F}\}|=q.$$

(ii) Observe that $ai^2+bi+c=aj^2+bj+c$ for $i\neq j$ if and only if 
$a(i^2-j^2)=a(i+j)(i-j)=-b(i-j)$, and thus if and only if $i+j=\frac{-b}{a}$, i.e
$j=-i+\frac{-b}{a}$. 
Thus given any $i\in \mathcal{F}$, we can find at most one other element $j\in\mathcal{F}$
such that $j\neq i$ and $ai^2+bi+c=aj^2+bj+c$, namely $j=i+\frac{-b}{a}$.
We conclude that the set  $\{ai^2+bi+c\mid i\in \mathcal{F}\}$
 has at least $\left\lceil\frac{q}{2}\right\rceil$ elements and the proof is complete. 
\end{proof}  
\begin{lem}\label{xysoln} 
Let $\mathcal{F}$ be a finite field with $q$ elements.
 Fix $a,b,c,d$ and  $e$ in $\mathcal{F}$.
Then for at least $q-1$ distinct values of $f$ in $\mathcal{F}$, the equation 
\begin{equation}\label{equationbasic}
ax^2-y^2+bxy+cy+(d-f)x+e=0
\end{equation}
\noindent has a solution $(x,y) \in \mathcal{F}\times\mathcal{F}$ with $x\neq 0$. 
In particular, the set
$\{\frac{1}{x}(ax^2-y^2+bxy+cy+e)+d\mid (x,y)\in \mathcal{F}\times \mathcal{F}, x\neq 0\}$ has at least $q-1$ elements. 
\end{lem}
\begin{proof}
Suppose $a = 0$. If $e\neq 0$, then $(\frac{-e}{d-f},0)$ is a solution for the equation as long as $f\neq d$.
If $e=0$, choose $y \neq 0$.  Then a solution for \eqref{equationbasic} is 
$(\frac{y^2}{by+c+d-f},y)$, for all $f \neq by+c+d$. 

We may assume then that $a\neq0$. 

{\bf Case 1.} 
Assume that $q$ is even.

Suppose $b=0$.  If $f=d$, choose $y$ such that $-y^2+cy+e\neq 0$. Then the set $\{ax^2-y^2+cy+e\mid x\in \mathcal{F}\}$ has $q$ elements by Lemma \ref{fieldsize} (i).  Hence, \eqref{equationbasic} has a solution with $x\neq 0$.

Suppose $f\neq d$.  For each $y$, choose $x=\frac{-cy}{d-f}$, so that $cy+(d-f)x=0$. By Lemma \ref{fieldsize} (i),  the set $\{ax^2-y^2+e 
\mid y\in \mathcal{F}, x=\frac{-cy}{d-f}\} = \{(\frac{-ac^2}{(d-f)^2} - 1)y^2 + e\mid y\in \mathcal{F}\}$ has $q$ elements, as long as $-ac^2 \neq (d-f)^2$. Then 
for some $y$, $(\frac{-cy}{d-f},y)$ is a solution for \eqref{equationbasic}, and $x\neq 0$ as long as $c\neq 0$ and $y \neq 0$. If $c=0$, then 
the set $\{ax^2-y^2+(d-f)x+e\mid y\in \mathcal{F}, x=1\}$ has $q$ elements by Lemma \ref{fieldsize} (i),
and so for some $y$, $(1,y)$ is a solution for \eqref{equationbasic}.
If $y=0$, then $e=0$, and so $(-\frac{d-f}{a},0)$ is a solution for \eqref{equationbasic} with $x\neq0$.

We may assume then that $b\neq 0$.  If $c\neq 0$, then
consider $x= \frac{-c}{b}$ and so $bxy+cy=0$. As before, the set
$\{ax^2-y^2+(d-f)x+e\mid y\in \mathcal{F}, x=\frac{-c}{b}\}$ has $q$ elements and thus
for some $y$, we have that $(\frac{-c}{b}, y)$ is a solution for \eqref{equationbasic}.
 We may assume then that $c=0$. Let $y=-\frac{d-f}{b}$ and thus $bxy+(d-f)x=0$. Then the set
$\{ax^2- y^2+ f\mid x\in \mathcal{F}\}$ has $q$ elements. Thus, for some 
$x$, we have that $(x, y)$ is a solution for \eqref{equationbasic}. 
If $x=0$ then $-y^2+e=0$ and since the field is of characteristic $2$, then there is a unique  $f\in \mathcal{F}$
such that $-(-\frac{d-f}{b})^2+e=0$. We conclude that in each case, for at least $q-1$ values of $f$, there
exists a solution for \eqref{equationbasic} with $x\neq 0$ when $q$ is even.

{\bf Case 2.} Assume that $q$ is odd.

Solving for $x$ with the quadratic formula, we get the discriminant
$$\Delta=(by+d-f)^2-4a(e+cy-y^2)=y^2(b^2+4a)+y(2b(d-f)-4ac)+((d-f)^2-4ae)
 $$
which takes on at least $(q+1)/2$ values as long as
$b^2+4a$ and $2b(d-f)-4ac$ are not both zero.

Suppose $b^2+4a = 0$ and assume that $c^2 + 4e$ is a square.

If $2b(d-f)-4ac = 0$, $(\frac{-c+2y+\sqrt{c^2+4e}}{b},y)$ is a
solution for \eqref{equationbasic}.  Observe that if for some $y$, $(0,y)$ is a solution, then 
$-y^2+cy+e=0$.  Thus $y = \frac{c\pm\sqrt{c^2+4e}}{2}$ and $(\frac{f-d-by}{a}, y)$ is another solution for \eqref{equationbasic}. Thus, for any value of $f$ such that 
$f\neq d+by$, there exists some $x\neq 0$ such that \eqref{equationbasic} holds.  

If $2b(d-f)-4ac \neq 0$, the discriminant $\Delta$ must be a square for some $y$, since $\frac{q+1}{2}$ elements of $\mathcal{F}$ are squares and thus $(\frac{-(by+d-f)+ \sqrt{\Delta}}{2a}, y)$ is a
solution for \eqref{equationbasic}.  If $(0,y)$ is a solution, then as above, $y = \frac{c\pm\sqrt{c^2+4e}}{2}$ and 
$(\frac{f-d-by}{a}, y)$ is another solution.  Thus, for $f\neq d+by$, there exists some $x\neq 0$ such that \eqref{equationbasic} holds.

Suppose $b^2+4a = 0$ and assume that $c^2 + 4e$ is not a square.  For all $f$ such that $2b(d-f)-4ac  \neq 0$,
the discriminant $\Delta$ must be a square for some $y$, and thus $(\frac{-(by+d-f)+ \sqrt{\Delta}}{2a}, y)$ is a solution for \eqref{equationbasic}.
If $(0,y)$ is a solution, then $-y^2+cy+e=0$ and so $c^2 + 4e$ is a square.  But $c^2 + 4e$ is not a square by assumption.
Hence, $(x, y)$ is a solution for \eqref{equationbasic} with $x = \frac{-(by+d-f)+ \sqrt{\Delta}}{2a} \neq 0$.

Now, suppose that $b^2+4a \neq 0$.  $\Delta$ must be a square for some $y$, and so $(\frac{-(by+d-f)+ 
\sqrt{\Delta}}{2a}, y)$ is a solution for \eqref{equationbasic}.  If $(0,y)$ is a solution, then $y = \frac{c\pm\sqrt{c^2+4e}}{2}$ and $(\frac{f-d-by}{a}, y)$ is another solution.  Thus, for $f\neq d+by$, there exists some $x\neq 0$ such that \eqref{equationbasic} holds.

We conclude that in each case,
for at least $q-1$ values of $f$, there
exists a solution $(x,y)$ for \eqref{equationbasic} with $x\neq0$. 

Observe that $\frac{1}{x}(ax^2-y^2+bxy+cy+e)+d=f$
 if and only if $ax^2-y^2+bxy+ cy+ (d-f)x+e=0$ for some $(x,y) \in \mathcal{F}\times\mathcal{F}$ with $x\neq 0$. The last statement
 then follows. 
\end{proof}  

{\bf Notation.} We will denote matrices with uppercase letters and elements in $\mathcal{F}$
with lowercase letters. 

\begin{rem}\label{conjugacytypes}
Let $A$ be an $n\times n$ matrix over $\mathcal{F}$. It is well known that 
$A$ is similar to a matrix $M$ such that $M$ is the direct sum of the companion matrices of a family of 
polynomials ${p_1}$, $\ldots$, ${p_t}$ in $\mathcal{F}[x]$.

Recall that the companion matrix of a polynomial $x^r+ a_{r-1} x^{r-1} +\cdots +
a_0 \in \mathcal{F}[x]$ is

\begin{equation}\label{rationalcanonical}
R=\left(\begin{array}{cccccccc} 0& 0 &0 & 0 & \cdots &0& 0 &-a_0\\
1& 0 &0 & 0 &\cdots &0 &0 &-a_1\\
0& 1 &0 & 0 &\cdots &0 &0 &-a_2\\
0& 0 &1 & 0 &\cdots &0 & 0 &-a_3\\
\vdots&    & & & & && \vdots\\
0& 0 &0 & 0 &\cdots &1& 0 &-a_{r-2}\\
0& 0 &0 & 0 &\cdots &0& 1 &-a_{r-1}
         \end{array} \right)
\end{equation}      
\end{rem}
\begin{hyp}\label{hypothesis}
Fix $n\geq 2$.
Let $M$ be a matrix similar to $A$ such that $M$ is the direct sum of the companion matrices of a family of 
polynomials ${p_1}$, $\ldots$, ${p_t}$. Also, let 
$N$ be a matrix similar to $B$ such that $N$ is the direct sum of the companion matrices of a family of 
polynomials ${q_1}$, $\ldots$, ${q_w}$. 
   Assume that $R$ is the last direct summand in $M$, that is,
 \begin{equation}
 M=\left(
\begin{array}{cc} M_{11} & 0\\
         0 & R\end{array} \right)
         \end{equation}
 \noindent  where $M_{11}$ is an $(n-r)\times (n-r)$ matrix. 
        
Let $S$ be the companion matrix of the polynomial $x^s+ b_{s-1} x^{s-1} +\cdots +
b_0$.  Assume that $S$ is the last direct summand in $N$, that is
 \begin{equation}\label{N}
         N=\left(
\begin{array}{cc} N_{11} & 0\\
         0 & S\end{array} \right),
         \end{equation} 
 \noindent where $N_{11}$ is an $(n-s)\times (n-s)$ matrix. 
 \end{hyp}

\begin{rem}\label{tracesargument}
Two matrices in the same conjugacy class have the same trace. Thus, if the matrices
do not have the same trace, then they belong to distinct conjugacy classes. 
\end{rem}

The following will be the main argument in the proofs of Theorem A and Theorem B and it follows from the previous 
remark.
\begin{lem}\label{mainargument}
  Let $\mathcal{H}$ be a subgroup of $\GL(n,q)$ and $A$, $B$ in $\mathcal{H}$.
  Suppose that the set $\{\trace(XY)\mid X\in A^{\mathcal{H}}B^{\mathcal{H}}\}$ has at least
  $r$ elements. 
 Then  $A^{\mathcal{H}}B^{\mathcal{H}}$ is the union of at least $r$ distinct conjugacy classes of $\mathcal{H}$, i.e.
 $\eta( A^{\mathcal{H}}B^{\mathcal{H}}  )\geq r$. 
\end{lem}

\begin{lem}\label{generalcase}
Let $R$ be an $r\times r$ matrix as in $\eqref{rationalcanonical}$ with $r\geq 2$, $I$ be the $(r-2)\times (r-2)$ identity matrix,
$D=\left(
\begin{array}{cc} a & b\\
         c & d\end{array} \right)$ be a $2\times 2$ matrix with determinant $w=ad-bc\neq 0$. 
 and $0$ be a matrix of  zeros with the appropriate size. 
 Set $E=\left(
\begin{array}{cc} I & 0\\
         0 & D\end{array} \right)$.
Then for $r=2$, 
$$R^E = \frac{1}{w}\left(\begin{array}{cc} -a_0cd -ab+a_1bc &-a_0d^2-b^2+a_1bd\\
a_0c^2 +a^2-aa_1c &a_0cd+ab-aa_1d\end{array} \right),$$
and for $r>2$, $R^{E}$ is the $r\times r$ matrix 
$$\left(\begin{array}{ccccccccc} 0& 0 &0 & 0 &\cdots&0 &0 & -a_0 c &-a_0d\\
1& 0 &0 & 0&\cdots&0&0 & -a_1 c &-a_1d\\
0& 1 &0 & 0 &\cdots&0 &0 & -a_2c  &-a_2d\\
\vdots&    &&&&&&\vdots & \vdots\\
0& 0 &0 & 0 &\cdots&1&0 & -a_{r-3}c & -a_{r-3}d \\
0& 0 &0 & 0 &\cdots&0&\frac{d}{w}  & \frac{-a_{r-2}cd -ab+a_{r-1}bc}{w} &\frac{-a_{r-2}d^2-b^2+a_{r-1}bd}{w}\\
0& 0 &0 & 0 &\cdots&0&\frac{-c}{w}  & \frac{a_{r-2}c^2 +a^2-aa_{r-1}c}{w} &\frac{a_{r-2}cd+ab-aa_{r-1}d}{w}
         \end{array} \right).$$

\end{lem}
\begin{proof}

Suppose $r=2$.  Then 
\begin{eqnarray*}
R^E & = &\frac{1}{w} \left(\begin{array}{cc} d & -b \\ -c & a\end{array}\right)\left(\begin{array}{cc} 0 & -a_0 \\ 1 & -a_1\end{array}\right)\left(\begin{array}{cc} a & b \\ c & d\end{array}\right)\\
& = & \frac{1}{w}\left(\begin{array}{cc} d & -b \\ -c & a\end{array}\right)\left(\begin{array}{cc} -a_0c & -a_0d \\ a-a_1c & b-a_1d\end{array}\right)\\
& = & \frac{1}{w}\left(\begin{array}{cc} d(-a_0c) -b(a-a_1c) & d(-a_0d) -b(b-a_1d)\\
-c(-a_0c) + a(a-a_1c) & -c(-a_0d) + a(b-a_1d)\end{array} \right).
\end{eqnarray*}

For $r>2$, 
$RE = \left(\begin{array}{cc} R_{11} & R_{12} \end{array}\right) 
\left(\begin{array}{cc} I & 0 \\ 0 & D \end{array}\right)
= \left(\begin{array}{cc} R_{11} & R_{12}D \end{array}\right)$, where $R_{11}$ and $R_{12}$ are $r \times (r-2)$ and $r \times 2$ submatrices of $R$, respectively.  Hence, 

$$RE = \left(\begin{array}{ccccccccc} 0& 0 &0 & 0 &\cdots&0 &0 & -a_0 c &-a_0d\\
1& 0 &0 & 0&\cdots&0&0 & -a_1 c &-a_1d\\
0& 1 &0 & 0 &\cdots&0 &0 & -a_2c  &-a_2d\\
\vdots&    &&&&&&\vdots & \vdots\\
0& 0 &0 & 0 &\cdots&1 &0& -a_{r-3}c & -a_{r-3}d \\
0& 0 &0 & 0 &\cdots&0&1 & -a_{r-2}c & -a_{r-2}d \\
0& 0 &0 & 0 &\cdots&0&0& a-a_{r-1}c & b-a_{r-1}d
         \end{array} \right).$$
 Observe that        
\begin{equation*}R^{E} = E^{-1}(RE) = \left(\begin{array}{cc} I & 0 \\ 0 & D^{-1} \end{array}\right)
\left(\begin{array}{c} (RE)_{11} \\ (RE)_{21} \end{array}\right)
= \left(\begin{array}{c} (RE)_{11} \\ D^{-1}(RE)_{21} \end{array}\right),\end{equation*}
\noindent  where $(RE)_{11}$ and $(RE)_{21}$ are $(r-2) \times r$ and $2 \times r$ submatrices of $RE$, respectively.
Hence, the first $r-2$ rows of $R^E$ are identical to those of $RE$, and the last two form the submatrix 
$$\left(\begin{array}{cccccc}
0 &\cdots&0&\frac{d}{w} &\frac{d(-a_{r-2}c) -b(a-a_{r-1}c)}{w} & \frac{d(-a_{r-2}d) -b(b-a_{r-1}d)}{w} \\
0 &\cdots&0&\frac{-c}{w}& \frac{-c(-a_{r-2}c) + a(a-a_{r-1}c)}{w} & \frac{-c(-a_{r-2}d) + a(b-a_{r-1}d)}{w}
         \end{array} \right).$$

\end{proof}

\begin{lem}\label{calculationsmain1}
Assume Hypothesis \ref{hypothesis}.
Let $E_1=\left(
\begin{array}{cc} I & 0\\
         0 & E\end{array} \right)$, where $E$ is as in Lemma \ref{generalcase}.  Given $x,y\in \mathcal{F}$ with $x\neq0$,  let $D(x,y)=\left(
\begin{array}{cc} x & y\\
         0 & 1\end{array} \right)$, i.e.
    $a=x$, $b=y$, $c=0$ and $d=1$ in $D$, and $E(x,y)=\left(
\begin{array}{cc} I & 0\\
         0 & D(x,y)\end{array} \right)$.  
         
(i) If $r>2$ and $s>2$, then
\begin{equation*}
\begin{array}{rcl}
\trace( M^{E_1} N)& = & \trace(MN)
   -a_{r-3}c + a_{r-2}  +b_{s-2} - a_{r-1}b_{s-1} \\ & &+\frac{1}{w}(-a_{r-2}d^2-b^2+a_{r-1}bd +b_{s-3}c - a_{r-2}b_{s-2}c^2 -a^2b_{s-2} \\& &+aa_{r-1}b_{s-2}c -a_{r-2}b_{s-1}cd -abb_{s-1} +aa_{r-1}b_{s-1}d).
 \end{array}
 \end{equation*}
In particular,
 \begin{equation*}
\begin{array}{rcl}
 \trace(M^{E(x,y)} N)& = &\frac{1}{x}(-b_{s-2}x^2-y^2-b_{s-1}xy+a_{r-1}y-a_{r-2})\\ & &+a_{r-2}  +b_{s-2} +\trace(MN).
  \end{array}
 \end{equation*}
  Thus the set $\{\trace(M^{E(x,y)}N)\mid (x,y)\in \mathcal{F}\times\mathcal{F}, x\neq0\}$ has at least $q-1$ elements, and $\{\trace(M^{E(1,y)}N)\mid y\in \mathcal{F}\}$ has at least   $\left\lceil\frac{q}{2}\right\rceil$ elements. 
 
 (ii) If $r=2$ and $s\geq r$, then
\begin{equation*}
\begin{array}{rcl}
\trace( M^{E_1} N) & = &\trace(MN)+a_0 +b_{s-2} -a_1b_{s-1} \\ & &+\frac{1}{w}(-a_0d^2-b^2+a_1bd -b_{s-2}(a_0c^2 +a^2-aa_1c) \\ & &-b_{s-1}(a_0cd+ab-aa_1d)).
\end{array}
\end{equation*}
In particular,
\begin{equation*}
\begin{array}{rcl}
\trace(M^{E(x,y)} N)&= &\frac{1}{x}(-b_{s-2}x^2-y^2 -b_{s-1}xy+a_1y -a_0))\\ & &+ a_0 +b_{s-2} +\trace(MN).
\end{array}
\end{equation*}
 
  Thus the set $\{\trace(M^{E(x,y)}N)\mid (x,y)\in \mathcal{F}\times\mathcal{F}, x\neq0\}$ has at least $q-1$ elements, and $\{\trace(M^{E(1,y)}N)\mid y\in \mathcal{F}\}$ has at least $\left\lceil\frac{q}{2}\right\rceil$ elements. 
    
 (iii) Assume that $r\geq 2$. Let $N_1= \left(
\begin{array}{cc} I & 0\\
         0 & D_1\end{array} \right)$, where $D_1=\left(
\begin{array}{cc} u & 0\\
         0 & v\end{array} \right)$ is a $2\times 2$ matrix where $u\neq v$.
      Then
\begin{equation*}
\begin{array}{rcl}
\trace( M^{E_1} N_1)&= &\trace(MN_1) + \frac{u}{w}(-a_{r-2}cd -ab+a_{r-1}bc)\\ & &+\frac{v}{w}(a_{r-2}cd+ab-aa_{r-1}d + a_{r-1}).
 \end{array}
 \end{equation*}   
Thus
\begin{equation*}
\begin{array}{rcl}
\trace(M^{E(x,y)} N_1)&= & \frac{1}{x}(xy(v-u)-a_{r-1}x+a_{r-1})+\trace(MN_1).
 \end{array}
 \end{equation*}
 Therefore the set  $\{\trace(M^{E(1,y)}N_1)\mid (x,y)\in \mathcal{F}\times\mathcal{F}, x\neq0\}$ has $q$ elements.
       \end{lem}
       \begin{proof}
       (i) Observe that all but the last three elements of the diagonal of 
       the matrices $MN$ and
       $M^{E_1}N$  have the same values. 
Hence, using the previous result, the last three diagonal values of $M^{E_1}N - MN = (M^{E_1} - M)N $ are
\begin{itemize}
\item $(0 - 0)0 + (-a_{r-3}c - 0)1 + (-a_{r-3}d + a_{r-3})0 = -a_{r-3}c$,
\item $(\frac{d}{w} - 1)0 + (\frac{1}{w}(-a_{r-2}cd -ab+a_{r-1}bc) - 0)0 + (\frac{1}{w}(-a_{r-2}d^2-b^2+a_{r-1}bd) + a_{r-2})1 = \frac{1}{w}(-a_{r-2}d^2-b^2+a_{r-1}bd) + a_{r-2}$, and
\item $(\frac{-c}{w} - 0)(-b_{s-3}) + (\frac{1}{w}(a_{r-2}c^2 +a^2-aa_{r-1}c) - 1)(-b_{s-2}) + (\frac{1}{w}(a_{r-2}cd+ab-aa_{r-1}d) + a_{r-1})(-b_{s-1}) = \frac{1}{w}(b_{s-3}c - a_{r-2}b_{s-2}c^2 -a^2b_{s-2} +aa_{r-1}b_{s-2}c -a_{r-2}b_{s-1}cd -abb_{s-1} +aa_{r-1}b_{s-1}d) +b_{s-2} - a_{r-1}b_{s-1}$.
\end{itemize}

Hence, 
\begin{eqnarray*}\trace(M^{E_1}N) - \trace(MN) & = & \trace(M^{E_1}N - MN) \\ & = &-a_{r-3}c + a_{r-2}  +b_{s-2} - a_{r-1}b_{s-1}\\& & +\frac{1}{w}(-a_{r-2}d^2-b^2+a_{r-1}bd +b_{s-3}c \\& &- a_{r-2}b_{s-2}c^2 -a^2b_{s-2} +aa_{r-1}b_{s-2}c \\& &-a_{r-2}b_{s-1}cd -abb_{s-1} +aa_{r-1}b_{s-1}d).\end{eqnarray*}

Thus, when $a=x$, $b=y$, $c=0$ and $d=1$, 
\begin{eqnarray*}\trace(M^{E_1}N) - \trace(MN) & = & a_{r-2}  +b_{s-2} - a_{r-1}b_{s-1}\\& & +\frac{1}{x}(-a_{r-2}-y^2+a_{r-1}y -b_{s-2}x^2  \\& & -b_{s-1}xy +a_{r-1}b_{s-1}x)\\
&=& \frac{1}{x}(-b_{s-2}x^2-y^2  -b_{s-1}xy +a_{r-1}y-a_{r-2})\\ & &+ a_{r-2} +b_{s-2}. 
\end{eqnarray*}
By Lemma \ref{xysoln}, the set 
$\{\frac{1}{x}(-b{s-2}x^2-y^2  -b_{s-1}xy +a_{r-1}y-a_{r-2})+ a_{r-2} +b_{s-2} +\trace(MN)\mid (x,y)\in \mathcal{F}\times\mathcal{F},
x\neq 0\}$ has at least $q-1$ elements. 

By Lemma \ref{fieldsize} (ii), the set $\{\trace(M^{E(1,y)}N\mid y\in \mathcal{F}\}=
\{-b_{s-2}-y^2 +(-b_{s-1}+a_{r-1})y -a_{r-2} + a_{r-2} +b_{s-2} + \trace(MN)\mid y\in \mathcal{F}\}$ has 
at least  $\left\lceil\frac{q}{2}\right\rceil$ elements.  

       (ii)  In this case, the diagonals of $MN$ and $M^{E_1}N$ have the same values, in all but the last two entries.  Hence, using the previous result, the last two diagonal values of $M^{E_1}N - MN = (M^{E_1} - M)N $ are
\begin{itemize}
\item $(\frac{1}{w}(-a_0cd -ab+a_1bc) - 0)0 + (\frac{1}{w}(-a_0d^2-b^2+a_1bd) + a_0)1$ and
\item $(\frac{1}{w}(a_0c^2 +a^2-aa_1c) - 1)(-b_{s-2}) + (\frac{1}{w}(a_0cd+ab-aa_1d) + a_1)(-b_{s-1})$.
\end{itemize}
Hence, $\trace(M^{E_1}N) - \trace(MN) = \trace(M^{E_1}N - MN) = a_0 +b_{s-2} -a_1b_{s-1} +\frac{1}{w}(-a_0d^2-b^2+a_1bd -b_{s-2}(a_0c^2 +a^2-aa_1c) -b_{s-1}(a_0cd+ab-aa_1d))$.

When $a=x$, $b=y$, $c=0$ and $d=1$, $\trace(M^{E_1}N) - \trace(MN) = a_0 +b_{s-2} -a_1b_{s-1} +\frac{1}{x}(-a_0-y^2+a_1y -b_{s-2}x^2 -b_{s-1}xy+a_1b_{s-1}x)$.
Thus $\trace(M^{E_1}N)= \frac{1}{x}( -b_{s-2} x^2-y^2 -b_{s-1}xy+a_1y -a_0) + a_0 +b_{s-2} + \trace(MN)$. By Lemma \ref{xysoln}
it follows that $\trace(M^{E_1}N)$ can take $q-1$ values.
 
By Lemma \ref{fieldsize} (ii), the set $\{\trace(M^{E(1,y)}N)\mid y\in \mathcal{F}\}=
\{-b_{s-2}-y^2 +(-b_{s-1}+a_1)y -a_0 + a_0 +b_{s-2} + \trace(MN)\mid y\in \mathcal{F}\}$ has 
at least  $\left\lceil\frac{q}{2}\right\rceil$ elements.

       (iii)  In this case, the diagonals of $MN_1$ and $M^{E_1}N_1$ have the same values, in all but the last two entries.  Hence, using the previous result, the last two diagonal values of $M^{E_1}N_1 - MN_1 = (M^{E_1} - M)N_1 $ are
\begin{itemize}
\item $(\frac{1}{w}(-a_{r-2}cd -ab+a_{r-1}bc) - 0)u + (\frac{1}{w}(-a_{r-2}d^2-b^2+a_{r-1}bd) + a_{r-2})0$ and
\item $(\frac{1}{w}(a_{r-2}c^2 +a^2-aa_{r-1}c) - 1)0 + (\frac{1}{w}(a_{r-2}cd+ab-aa_{r-1}d) + a_{r-1})v$.
\end{itemize}
Hence, $\trace(M^{E_1}N_1) - \trace(MN_1) = \trace(M^{E_1}N_1 - MN_1) = \\ \frac{u}{w}(-a_{r-2}cd -ab+a_{r-1}bc)+\frac{v}{w}(a_{r-2}cd+ab-aa_{r-1}d + a_{r-1})$.

When $a=x$, $b=y$, $c=0$ and $d=1$, $\trace(M^{E_1}N_1) - \trace(MN_1) = \frac{-uxy}{x}+\frac{v}{x}(xy-a_{r-1}x + a_{r-1})$.
Thus $\trace(M^{E_1}N_1)= \frac{1}{x}(xy(v-u)-a_{r-1}x+a_{r-1})+\trace(MN_1)$. 
Since $v-u\neq 0$, the set $\{\frac{1}{x}(xy(v-u)-a_{r-1}x+a_{r-1})+\trace(MN_1)\mid x=1, y\in \mathcal{F}\}$ has $q$ elements. 
       \end{proof}

       \begin{lem}\label{calculationsmain2} 
Let $C=\left(\begin{array}{cc}  C_{11}& 0\\
         0 & D_1\end{array} \right)$ be a $n\times n$ matrix,  where $D_1=\left(\begin{array}{cc}  u_1& 0\\
         0 & v_1\end{array} \right)$ is a $2\times2$ matrix where
         $u_1\neq v_1$.
          Let $E=\left(\begin{array}{cc}  I & 0\\
         0 & D \end{array} \right)$ be in $\GL(n,q)$, where $D=\left(\begin{array}{cc} a & b\\c&d\end{array}\right)$
         and $ad-bc=1$. 
         
         Then $C^{E_1}=\left(\begin{array}{cc}  C_{11}& 0\\
         0 & D_1^{D}\end{array}\right)$, where
         $D_1^{D}=\left(
\begin{array}{cc} adu_1-bcv_1& bd(u_1-v_1)\\
         -ac(u_1-v_1) &adv_1-bcu_1\end{array} \right)$. 
         
         Thus, given a matrix $N=\left(\begin{array}{cc}  N_{11}& 0\\
         0 & D_2\end{array} \right)$ in $\GL(n,q)$, where $D_2=\left(\begin{array}{cc}  u_2& 0\\
         0 & v_2\end{array} \right)$ is a $2\times2$ matrix where
         $u_2\neq v_2$, we have that
         $$
         \trace(C^{E} N)= \trace(CN)-u_1u_2-v_1v_2+u_2(adu_1-bcv_1)+v_2(adv_1-bcu_1).
        $$
      In particular, if we fix $x\in \mathcal{F}$, $ad=x$, and $bc=x-1$, we have that
 $$\trace(C^{E} N)= x(u_1-v_1)(u_2-v_2)+ ( \trace(CN) - (u_1-v_1)(u_2-v_2) ).$$
       Therefore given any $f\in \mathcal{F}$, we can find some $x\in \mathcal{F}$ such that
       $\trace(C^{E} N)=f$. 
         \end{lem}
\begin{proof}     Observe that     
\begin{eqnarray*} 
D_1^D & = & \left(\begin{array}{cc} d & -b \\ -c & a\end{array}\right)\left(\begin{array}{cc} u_1& 0\\
         0 & v_1\end{array}\right)\left(\begin{array}{cc} a & b \\ c & d\end{array}\right)\\
& = & \left(\begin{array}{cc} d & -b \\ -c & a\end{array}\right)\left(\begin{array}{cc} au_1 & bu_1 \\ cv_1 & dv_1\end{array}\right)\\
& = & \left(\begin{array}{cc} adu_1-bcv_1& bd(u_1-v_1)\\
         -ac(u_1-v_1) &adv_1-bcu_1\end{array} \right).
\end{eqnarray*}
         
The diagonals of $C^{E} N$ and $CN$ have the same values, in all but the last two entries.  Hence, the last two diagonal values of $C^{E}N - CN = (C^{E} - C)N $ are
\begin{itemize}
\item $(adu_1-bcv_1 - u_1)u_2 + (bd(u_1-v_1) - 0)0$ and
\item $(-ac(u_1-v_1) - 0)0 + (adv_1-bcu_1 - v_1)v_2$.
\end{itemize}
Hence, 
\begin{equation*}
\begin{array}{rcl}
\trace(C^{E}N)  - \trace(CN) & = &\trace(C^{E}N - CN) \\ &= &(adu_1-bcv_1 - u_1)u_2 + (adv_1-bcu_1 - v_1)v_2.
\end{array}
\end{equation*}

When $ad=x$ and $bc=x-1$, 
\begin{equation*}
\begin{array}{rcl}
 \trace(M^{E}N) - \trace(MN)   &= & (xu_1-(x-1)v_1 - u_1)u_2 \\ & &+ (xv_1-(x-1)u_1 - v_1)v_2  \\ 
 & = &x((u_1-v_1)u_2 + (v_1 - u_1)v_2) \\ & &+ (v_1 - u_1)u_2 + (u_1 - v_1)v_2\\
 &= &x(u_1-v_1)(u_2-v_2) - (u_1-v_1)(u_2-v_2).         
\end{array}
\end{equation*} 
Since $u_1\neq v_1$ and $u_2\neq v_2$, we have that $(u_1-v_1)(u_2-v_2)\neq 0$ and thus the set 
 $ \{ x(u_1-v_1)(u_2-v_2)+ ( \trace(CN) - (u_1-v_1)(u_2-v_2) )\mid x\in \mathcal{F}\}=\mathcal{F}$. 
\end{proof}
 \begin{proof}[Proof of Theorem  A]
 If at least one of $A$, $B$ is in the center ${\bf Z}(\mathcal{G})$ of $\mathcal{G}$, then $A^{\mathcal{G}}B^{\mathcal{G}}=(AB)^{\mathcal{G}}$. Thus we
 may assume that both matrices $A$ and $B$ are not in the center, that is we may assume that
 both $A$ and $B$ are non-scalar matrices. 
 
 Assume Hypothesis \ref{hypothesis}. 
 If $r\geq 2$ or $s\geq 2$, then by Lemma \ref{interchange}, Lemma \ref{mainargument} and 
       Lemma \ref{calculationsmain1} we have that 
        $\eta(A^{\mathcal{G}} B^{\mathcal{G}})\geq q-1$.
        Without loss of generality, we may assume then that for $i=1,\ldots, t$ 
 and $j=1, \ldots,w$,  the polynomials
 $p_i$, $q_j$ have degree 1, that is both $A$ and $B$ are diagonal matrices. Since 
 both $A$ and $B$ are non-scalar matrices, we may assume that $A$ is similar to $C$ and $B$ is similar to $N$, where $C$ and $N$ are as constructed in Lemma \ref{calculationsmain2}.  Hence, by 
 Lemma \ref{calculationsmain2},
  $\eta(A^{\mathcal{G}} B^{\mathcal{G}})\geq q$. 
 \end{proof}
 
 \begin{rem}\label{optimal}
 Let $\mathcal{F}$ be a field with $q = 2^m$ elements, for some integer $m>1$. Set $\mathcal{G}=\GL(2,q)=\GL(2,\mathcal{F})$. 
Let $A=\left(\begin{array}{cc} 1 & 1\\
         0 & 1\end{array} \right)$ and $B=\left(\begin{array}{cc} 0 & 1\\
         -1 & w\end{array} \right)$ in $\mathcal{G}$, where $x^2-wx+1$ is an irreducible polynomial over 
         $\mathcal{F}$. Observe that both $A$ and $B$ are in $\SL(2,q)$ and thus 
         $A^{\GL(2,q)} B^{\GL(2,q)}\subseteq \SL(2,q)$.
         
         By Proposition 2.13 of \cite{sl2q}, $\eta(A^{\SL(2,q)}B^{\SL(2,q)})=q-1$. Since $x\mapsto x^2$ is an automorphism of 
         $\mathcal{F}$, two matrices $C$, $D$ in $\GL(2,q)$ are similar if and only there 
         exists some $H\in \SL(2,q)$ such that $C^H=D$. Thus $\eta(A^{\SL(2,q)}B^{\SL(2,q)})=\eta(A^{\mathcal{G}}B^{\mathcal{G}})=q-1$.
         
         
         \end{rem}

  \begin{proof}[Proof of Theorem  B]
  As in the proof of Theorem A, we may assume that 
 both matrices $A$ and $B$ in $\mathcal{S}=\SL(n,q)$ are not in the center, that is we may assume that
 both $A$ and $B$ are non-scalar matrices. 
 
 We may assume then Hypothesis \ref{hypothesis}. Observe that the matrix $E(1,y)$ in Lemma \ref{calculationsmain1}
 is in $\mathcal{S}$. Thus if $r\geq 2$ or $s\geq 2$, then by Lemma \ref{interchange}, Lemma \ref{mainargument} and 
       Lemma \ref{calculationsmain1} we have that 
        $\eta(A^{\mathcal{S}} B^{\mathcal{S}})\geq \left\lceil\frac{q}{2}\right\rceil  $. As in the proof of
        Theorem A, we may assume then 
        that both $A$ and $B$ are diagonal matrices. Observe that since $ad=x$ and $bc=x-1$, then 
        $ad-bc=1$ and so the matrix $E$ is in $\mathcal{S}$. 
        Since $A$ and $B$ are non-scalar matrices, the result then follows by Lemma \ref{mainargument} and  
       Lemma \ref{calculationsmain2}. 
  \end{proof} 
  \end{section}


\begin{thebibliography}{9}
\bibitem{symmetric}
E. Adan-Bante, H. Verrill, Symmetric groups and conjugacy classes, J. Group Theory 11 (2008), no. 3, 371-379.
 
   
   \bibitem{sl2q}
 E. Adan-Bante, J. M. Harris,  On conjugacy classes of $\SL(2,q)$, preprint.
   \bibitem{GAP4}
  The GAP~Group, \emph{GAP -- Groups, Algorithms, and Programming, 
  Version 4.4.10}; 
  2007,
  \verb+(http://www.gap-system.org)+. 
  
  
 
\end{thebibliography}
\end{document}